\documentclass{amsart}
\newcommand{\red}{\operatorname{red}}
\newcommand{\val}{\nu_\pi}
\newcommand{\smU}{U^0}
\newcommand{\supp}{\operatorname{supp}}
\newcommand{\ord}{\operatorname{ord}}
\newcommand{\Kbar}{\bar{K}}
\newcommand{\Xgrph}{\Gamma(X)}
\newcommand{\Qp}{\Q_p}
\newcommand{\Q}{\mathbb{Q}}

\newcommand{\PP}{{\mathbb{P}}}

\newcommand{\res}{\operatorname{Res}}

\newcommand{\dlog}{\operatorname{d\!\log}}

\newcommand{\pair}[1]{{\left\langle #1 \right\rangle}}
\newcommand{\gpair}[1]{\pair{#1}_{\gl}}

\renewcommand{\O}{\mathcal{O}}

\newcommand{\gl}{{\textup{gl}}}

\def\htt_#1{H_#1^\otimes}

\newcommand{\col}{\textup{Col}}

\def\acoln_#1{\O_{\col,#1}}
\def\Ocoln_#1{\Omega_{\col,#1}^1}

\newcommand{\XX}{\mathcal{X}}

\renewcommand{\O}{\mathcal{O}}
\newcommand{\Z}{\mathbb{Z}}

\newcommand{\gipair}[1]{\pair{#1}_{\gl,U_v}}
\newcommand{\Ft}{\tilde{F}}

\newtheorem{theorem}{Theorem}[section]

\newtheorem{lemma}[theorem]{Lemma}
\newtheorem{corollary}[theorem]{Corollary}
\theoremstyle{definition}

\numberwithin{equation}{section}

\begin{document}
\title{Vologodsky integration on curves with semi-stable reduction}
\author{Amnon Besser}
\address[Besser]{Department of Mathematics, Ben Gurion University, Be'er Sheva 84105, Israel}
\email{bessera@math.bgu.ac.il}

\author{Sarah Livia Zerbes}
\address[Zerbes]{Department of Mathematics \\
University College London\\
Gower Street, London WC1E 6BT, UK}
\email{s.zerbes@ucl.ac.uk}

\begin{abstract}
We prove that the Vologodsky integral of a mermorphic one-form on a curve over a $p$-adic field with semi-stable reduction restrict to Coleman integrals on the rigid subdomains reducing to the components of the smooth part of the special fiber and that on the connecting annuli the differences of these Coleman integrals form a harmonic cochain on the edges of the dual graph of the special fiber. This determines the Vologodsky integral completely. We analyze the behavior of the integral on the connecting annuli and we explain the results in the case of a Tate elliptic curve.
\end{abstract}

\maketitle

\section{Introduction}
\label{sec:intro}

Coleman integration~\cite{Col82,Col-de88} is a method for defining iterated integrals on certain $p$-adic rigid analytic domains with good reduction. Unlike the complex case, the resulting integral is single valued. Vologodsky integration~\cite{Vol01} also produces iterated integrals but on algebraic varieties over the same fields, without any assumption on the reduction. They are known to be the same in the good reduction case. Vologodsky integration of holomorphic forms (without iteration) was known before Vologodsky's work~\cite{Zar96,Colm96} using the logarithm on the Albanese variety and generalizations. It is now also sometimes called \emph{abelian integration}.

The existence of Vologodsky integration might be a bit surprising from the point of view of Coleman integration because one can sometimes glue an algebraic variety out of several domains with good reduction, in which case one can do Coleman integration on each domain and try to glue the integrals together. This in general produces a multi valued integral, hence not the Vologodsky integral.

Our work set out to try to clarify the relation between the two integration theories starting with the simplest possible non-trivial case, that of a curve with semi-stable reduction. In the general iterated case some progress has been made over the last few years but the project is far from finished. On the other hand, the case of abelian integration was fairly easy to handle. Since then, it has proved useful for various problems, especially in the work of the first named author. Thus, a need for an account for the proof of this special case has arisen, The present work provides such an account.

We recall the setup of~\cite[Section~\ref{sec:globsemi}]{Bes17}. Let $K$ be a finite extension of $\Qp$ with ring of integers $\O_K$ and residue field $k$. Let $X$ be a curve over $K$ which is the generic fiber of a proper $\O_K$ scheme $\XX$ with semi-stable reduction
\begin{equation}
  \label{eq:Yi}
  T = \cup_i T_i\;.
\end{equation}
In particular, locally near an intersection point $T_i\cap T_j$ there are coordinates $x,y$ satisfying
\begin{equation}
  xy=\pi\;,\; T_i=(x)\;,\;T_j=(y) \label{eq:ss}
\end{equation}
(here, $(f)$ denotes the divisor of the rational function $f$). For simplicity we will assume that components $T_i$ and $T_j$ intersect at at most one point. We can easily get to this by blowing up and the main theorem will apply without this assumption.

Let $\Xgrph$ be the dual graph of $T$ with vertices $V$ and edges $E$ (this is of course an abuse of notation as it really depends on the particular model). The vertices correspond to the components $T_v$ while the edges are ordered pairs of intersecting components $(T_v,T_w)$ oriented from $v$ to $w$, so that an edge $e$ has tail $e^+=v$ and head $e^-=w$.

The reduction map $X\to T$ allows us to split $X$ into rigid analytic domains $U_v = \red^{-1} T_v$ which are  wide open spaces in the sense of Coleman. These then intersect along annuli corresponding bijectively to the unoriented edges of $\Xgrph$. Indeed, in terms of the coordinates $x,y$ appearing in~\eqref{eq:ss} the annulus corresponding to the edge $(T_i,T_j)$ gets mapped via $x$ (or $y$) to the rigid analytic space $ A(|\pi|,1)$ with
\begin{equation}
  \label{eq:annulus}
  A(r,s) := \{z\in \Kbar\;,\; r<|z|<s\}\;.
\end{equation}
 An orientation of an annulus fixes a sign for the residue along this annulus and we match oriented edges with oriented annuli as in~\cite[Def.~\ref{convention1}]{Bes17}. We use the same notation for the edge and for the associated oriented annulus.

Fix a branch $\log$ of the $p$-adic logarithm. Let $\omega$ be a meromorphic form on $X$. The Vologodsky integral (with some choice of a constant of integration) $F$ of $\omega$ is a function $F:X(K) \to K$ (note that in~\cite{Vol01}, and also in~\cite{Ber07}, the integral is made to depend on a universal $\log$, which we substitute with our particular choice). On the other hand, we can use Coleman integration to define another such function as follows. Let $\smU$ be the inverse image, under the reduction map $\red$, of the smooth part of $T$. Choose Coleman integrals $F_v$ for $\omega$ on $U_v$ for each $v\in V$. The $F_v$ give a well defined function $\Ft: \smU \to K$ and, by abuse of notation, a function $\Ft: X(K) \to K$ because $K$ points always reduce to the smooth part (one needs to restrict the choices of constants of integration to get a function into $K$ but we will ignore this minor point for ease of exposition). This choice of $\Ft$ has $|V|$ degrees of freedom, which we would like to restrict to just one degree of freedom if we are to get a function comparable to $F$. To that end we observe that when associating
\begin{equation}
  c(e)=c_\omega(e)=F_\omega^{e^-}|_{e} - F_\omega^{e^+}|_{e}\label{eq:comega}
\end{equation}
the $c(e)$ are functions on the corresponding annuli, but they are constant because both Coleman integrals differentiate to $\omega$. As any cochain on $E$ decomposes uniquely into a harmonic cochain and an exact cochain \cite[Theorem~\ref{thmharmonic}]{Bes17} there is a unique, up to constant, way of choosing the $F_v$ in such a way that $c$ is harmonic. This therefore defines $\Ft$ uniquely up to a constant. The goal of this article is to prove the following result.
\begin{theorem}\label{mainsarah}
  The function $F-\Ft$ is constant on $X(K)$.
\end{theorem}
In particular, Vologodsky integration is ``locally'' given by Coleman integration. Even this fact is not obvious. Note that the cochain $c_\omega$ of~\eqref{eq:comega} appearing in the construction of $\Ft$ is canonically associated with $\omega$. See~\cite[Prop.~\ref{creps}]{Bes17} for a cohomological interpretation of $c_\omega$.

By extending scalars, Vologodsky integration gives a function $F:X(\Kbar) \to \Kbar$. In particular, it is defined on points of the annuli $e\in E$. It clearly can not be defined there by the Coleman integrals $F_v$ because these do not even agree on the annuli. This, however,  does not contradict the result! In order to get the values of $F$ on points in the annuli one must make a ramified extension of $K$, say or ramification index $m$. Changing scalars to the integral model $\XX$ does not yield a semi-stable model and one needs to blow up the singular points of the special fiber. The resulting dual graph is obtained by taking $\Xgrph$ and subdividing each edge into $m$ edges, the additional vertices correspond, under the identification above with an annulus $A(|\pi|,1)$~\eqref{eq:annulus} with the subdomains $A(|\pi|^{(k-1)/m},(|\pi|^{(k+1)/m})$ for $k=1,\ldots,m-1$. As a harmonic cochain on such a graph  must give the same value to all the edges obtained from subdividing a single edge of $\Xgrph$ the following interesting Corollary follows easily.
\begin{corollary}\label{interpolate}
  Let $\omega$ be as in Theorem~\ref{mainsarah} and let $c_\omega$ be the associated harmonic cochain. Let $e$ be an edge of $\Xgrph$ connecting vertices $v$ and $w$ where the corresponding annulus, still denoted $e$, is isomorphic via a coordinate $x$ to $A(|\pi|,1)$ Then, on $e(\Kbar)$ we have
  \begin{equation*}
    F=F_v+ c_\omega(e)\cdot \val\circ x
  \end{equation*}
  with $\val$ the valuation normalized so that $\val(\pi)=1$.
\end{corollary}
The phenomenon of a linear factor in the valuation appearing in the formula for the abelian integral on an annulus was observed independently by Stoll~\cite{Sto16} and proved by him and also in~\cite{Ka-Ra-Zu16}.

After proving the main result in Section~\ref{main} we discuss in Section~\ref{Tate} the case of a Tate elliptic curve, to demonstrate how the Vologodsky integral of a holomorphic form on a proper curve is independent of the choice of the branch of the logarithm even though the Coleman integrals do depend on it.

We note the following about the logical dependence of this work with~\cite{Bes17}. The current work relies heavily on~\cite{Bes17}, in particular on Section~\ref{sec:globsemi}. In turn, the result here is used (only) in Section~\ref{sec:end} of~\cite{Bes17}. This is done so that while the two papers reference each other there is no danger of a cyclic argument.

We would like to thank Wies\l awa Nizio\l, Rebecca Bellovin and Jessica Fintzen for helpful discussions regarding this work. The first named author was supported by Israel Science Foundation grant No. 1517/13. He would like to thank the department of Mathematics at the Georgia Institute of Technology, where a substantial part of the work on this paper was done.

\section{Proof of the main theorem}
\label{main}

\begin{proof}
We begin by noting that we may prove the result over a finite extension. This is clear for an unramified extension and for a ramified extension one needs to reverse the  argument leading to Corollary~\ref{interpolate}. This argument also shows that the result applies without assuming components only intersect at one point. Next we note that the integral of a form $\omega=df$, where $f$ is a rational function on $X$, is just $f$ up to a constant and the integral of $\omega=df/f$ is $\log(f)$. As this is the case for Coleman integration as well we get that the Vologodsky integral is given by Coleman integrals on the $U_v$. Furthermore, for both of these types of forms $\omega$ we have $c_\omega=0$, which is harmonic, proving the result in these two cases.

As usual we evaluate functions on divisors $G(\sum n_i P_i) = \sum n_i G(P_i)$. It is clear that to show the result it suffices to prove that $F(D)=\Ft(D)$ for any divisor of degree $0$. By the above remarks is suffices to prove this under the assumption that $D$ splits as a sum of $K$ points and $\omega$ is regular on the support of $D$. Let  $\alpha(D)= F(D)-\Ft(D)$. We need to show that $\alpha(D)=0$ for any divisor of degree $0$. We claim that it suffices to show this for a principal $D$. Indeed, if this is the case, then $\alpha$ factors via the Jacobian $J$ and gives an additive map $\alpha: J(K) \to K$. But, as the derivative of both integrals with respect to any of the points in $D$ is the same, namely $\omega$, this map will be locally constant hence $0$.

Thus, let $D = (f)$ be the divisor of a rational function. As noted before, we assume that $\omega$ is non-singular on the support of $D$, and we further assume, by extending the field of definition, that the support $\supp(\omega)$ of (the divisor of) $\omega$ splits as a union of $K$-rational points.  Recall from~\cite[Section~\ref{sec:double}]{Bes17} the definitions of the local pairings on points and annuli, and their global versions (this is an easy digest of results of~\cite{Bes98b,Bes00})
Consider the global pairing
 \begin{equation*}
   \gpair{F,\log(f)}=\sum_x \pair{F,\log(f)}_x
\end{equation*}
which is $0$ by~\cite[Prop.~3.6]{Bes00}. By assumption, there are no common singular points to $F$ and $\log(f)$, and separating the sum into a sum on each type of singular points and using~\cite[Def.~\ref{doubdef}]{Bes17} gives
\begin{equation*}
  0 = F(D) -\sum_{x\in \supp(\omega)} \res_x \log(x) \omega
\end{equation*}
Consider now the similarly defined global pairing
$  \sum_x \pair{\Ft,\log(f)}_x$.
As $d \Ft = \omega$ we get as before
\begin{equation*}
  \sum_x \pair{\Ft,\log(f)}_x =  \Ft(D) -\sum_{x\in \supp(\omega)} \res_x \log(x) \omega\;.
\end{equation*}
It is thus clear that to show $F(D)=\Ft(D)$ it suffices to show that
\begin{equation*}
  \sum_x \pair{\Ft,\log(f)}_x = 0\;.
\end{equation*}
To this end we use the expression for the left hand side found in Theorem~\ref{globcupthm} of~\cite{Bes17}. For this we note that $\log(f)$ is a Coleman integral of $\dlog f$ on each $U_v$, giving a collection of such Coleman integrals with associated cochain $c_{\dlog f}=0$. This gives
\begin{align*}
  \sum_x \pair{\Ft,\log(f)}_x &= 
 \sum_v \gipair{\omega,\dlog f}
 +  \sum_{e\in E/\pm }\left( c_\omega(e) \res_e \dlog f  -c_{\dlog f}(e)\res_e \omega\right)\\
 &=\sum_v \gipair{\omega,\dlog f}
 +  \sum_{e\in E/\pm } c_\omega(e) \res_e \dlog f \;.
\end{align*}
with $\gipair{~,~}$ the global pairing on the $U_v$ and $\sum_{e\in E/\pm}$ denotes a sum over the unoriented edges of an expression that is independent of the orientation.
For each vertex $v$ we have $\gipair{F_,\log(f)}= 0$, this time by~\cite[Cor.~4.11]{Bes98b}. By assumption, the cochain $c_\omega$ is harmonic and by the following Lemma the cochain $e\to \res_e \dlog(f)$ is exact. Therefore, the sum $ \sum_{e\in E/\pm } c_\omega(e) \res_e \dlog f$ is $0$, finishing the proof.
\end{proof}
\begin{lemma}
  The cochain $e\to \res_e \dlog(f)$ is the boundary of $v\mapsto \ord_{T_v} f$, where the last expression means the multiplicity of the component $T_v$ in the divisor of $f$.
\end{lemma}
\begin{proof}
By assumption $f$ has no zeros or poles on any annulus. We work locally near a singular point of $T$ where we have, as in~\eqref{eq:ss} coordinates $x,y$ with $xy=\pi$ and $x$ and $y$ define the two components intersecting at the point. We want to compare the residue on the annulus $e=\{|\pi|<|x|< 1\} $ to the difference of the orders of $f$ on the two components. Because this is clear for both $f=x$ and $f=y$ we can assume that the the divisor of $f$ does not include either component This means
that it is the quotient of polynomials $P(x,y)/Q(x,y)$ where both $P$ and $Q$ have the same property, so it is enough to prove for $f=P(x,y)$. Replacing $y$ by $\pi/x$ we get a Laurent polynomial
$f=\sum_{i=-m}^n a_i x^i$ with the properties $a_i \in \O_K$ and $i+\val(a_i)\ge 0$. Not being divisible by either $x$ or $y$ in this ring means that we have an $i\ge 0$ with $\val(a_i)=0$ and an $i\le 0$ such that $\val(a_i)=-i$. In terms of the Newton polygon this means that it is above the graph of
\begin{equation*}
  h(t)=
  \begin{cases}
    -t & t\le 0\\
    0 & t\ge 0
  \end{cases}
\end{equation*}
and touches it both for some negative $t$ and some positive $t$. 
By assumption there are no zeros of $f$
on the annulus $e$, implying the Newton polygon has no slopes strictly between  $-1$ and $0$. It is easy to see that this implies that $\val(a_0)=0$. Thus, the number of non-negative slopes, which is the number of roots with non-positive valuation, is exactly $n$, which is the order of pole at $\infty$. It follows easily from Coleman's ``Cauchy's Theorem''~\cite[Prop.~2.3]{Col89} that the residue of $\dlog f$ on an annulus $e$ on $\PP^1$, for a rational function $f$ whose divisor is disjoint from $e$, equals degree of the part of $(f)$ inside, or outside the annulus. Thus, in our case this residue is $0$ as required.
\end{proof}

\section{The case of a Tate elliptic curve}
\label{Tate}

Let $K$ be as in the introduction and let $q\in K^\times$ with valuation $n=\val(q)>0$. The Tate elliptic curve associated to $q$ is $E_q$ with $K$ points $K^\times /q^{\Z}$. We will assume that $n\ge 3$ so that the reduction is of the type considered in most of the text. Let $\omega = dz/z$ be the invariant differential. Its Vologodsky integral is the logarithm for the Tate curve, i.e., the unique homomorphism $E_q(K) \to K$ with differential $\omega$. This is clearly induced by the branch $\log_q$ of the logarithm, the one that sends $q$ to $0$. As this is the unique branch that factors via $E_q(K)$.

Let us see how this is obtained from Theorem~\ref{mainsarah}. The dual graph of the reduction of $E_q$ is an $n$-gon. We identify the vertices with the elements of the additive group $\Z/n\Z$ and in turn with the set $\{0,\ldots,n-1\} $. The preimage under the reduction map $\red$ of the smooth part of the component $v\in\Z/n\Z $ is the space $\val(z)=v$, mapping isomorphically onto its image in $E_q$ (which is the same for $v$'s congruent modulo $n$), and the wide open space $U_v$ is the space $v-1< \val(z)<v+1$, again identified with its image. The annulus corresponding to the pair $(v,v+1) $ is given by $v<\val(z)<v+1 $.

 We now fix any branch $\log$ of the logarithm. The Coleman integral of $\omega$ on $U_v$ is $\log$ up to a constant. We can choose the integral to be  $F_v=\log$ on each $U_v$, $v=0,\ldots, n-1$, making the differences $c_\omega(v,v+1)=F_{v+1}-F_v= 0$ on the annulus $(v,v+1)$ for $v=0,\ldots, n-2$. However, the edge $(n-1,0)$ is different: It corresponds to the image of the annulus $n-1<\val(z)<n $, which is the same as the image of the annulus $-1<\val(z)<0$, the two being identified via multiplication by $q$. Thus, for such a $z$ we have $F_{n-1}(z) =\log(z)$ while $F_0(z) = \log(z/q) = \log(z)-\log(q)$ so $c_\omega(n-1,0)= -\log(q)$.

 For the $n$-gon a harmonic cochain is a constant function while a cochain $b$ is exact if and only if $\sum b(e)=0$. Thus the decomposition of the resulting $c_\omega$ into harmonic $+$ exact is such that the harmonic is the constant $-\log(q)/n$ and the exact is $d \gamma$, with $\gamma$, normalized so that $\gamma(0)=0$, has $\gamma(v)= \log(q)\cdot v/n$. To get the choices of Coleman integrals have differences forming a harmonic cochain we need to take on $\val(z)=v$
 \begin{equation*}
   \Ft_v(z) = F_v(z) - \gamma(v) = \log(z)-\frac{v\log(q)}{n} = \log(z)-\val(z) \frac{\log(q)}{n} = \log_q(n)
 \end{equation*}
so that the Vologodsky integral of $\omega$ is indeed $\log_q$.

\end{document}